\documentclass[11pt]{amsart}
\usepackage{graphicx,color,amscd,amsmath,amssymb,verbatim}
\usepackage{amssymb}
\newtheorem{thm}{Theorem}[section]
\newtheorem{lem}[thm]{Lemma}
\theoremstyle{definition}
 \theoremstyle{remark}
\newtheorem{rem}[thm]{Remark}
\numberwithin{equation}{section}

\begin{document} \title[]{Spectral properties of weighted composition operators on the Bloch and Dirichlet spaces}

\author[Eklund] {Ted Eklund}\address{Ted Eklund. Department of Mathematics, \AA bo Akademi University. FI-20500 \AA bo, Finland.
\emph{e}.mail: ted.eklund@abo.fi}
\author[Lindstr{\"o}m] {Mikael Lindstr\"om}\address{Mikael Lindstr{\"o}m. Department of Mathematics, \AA bo Akademi University. FI-20500 \AA bo, Finland. \emph{e}.mail: mikael.lindstrom@abo.fi} 
\author[Mleczko] {Pawe{\l} Mleczko}\address{Pawe{\l} Mleczko. Faculty of Mathematics and Computer Science, Adam Mickiewicz University in Pozna\'{n}, Umultowska 87, 61-614 Pozna\'{n}, Poland.
\emph{e}.mail: pml@amu.edu.pl}

\begin{abstract}
The spectra of invertible weighted composition operators $uC_\varphi$ on the Bloch and Dirichlet spaces are studied. In the Bloch case we obtain a complete description of the spectrum when $\varphi$ is a parabolic or elliptic automorphism of the unit disc. In the case of a hyperbolic automorphism $\varphi$, exact expressions for the spectral radii of invertible weighted composition operators acting on the Bloch and Dirichlet spaces are derived.
\end{abstract}

\maketitle

\section{Introduction}
   
The space of analytic functions on the open unit disk $\mathbb D$ in the complex plane $\mathbb{C}$ is denoted by $H(\mathbb D)$. Every analytic selfmap $\varphi \colon \mathbb D \to \mathbb D$ of the unit disc induces a \emph{composition operator} $C_{\varphi}f = f \circ \varphi$ on $H(\mathbb{D})$. These operators have been studied for many decades starting from the papers of Littlewood, Hardy and Riesz in the beginning of the 20th century. For general information of composition operators on classical spaces of analytic functions the reader is referred to the excellent monographs by Cowen and MacCluer \cite{4} and Shapiro \cite{11}. In recent years this well-recognized theory has received new stimulus from the more general situation of linear \emph{weighted composition operators} $uC_{\varphi}(f) = u(f \circ \varphi)$, where $u \in H(\mathbb{D})$. The main objective when studying the operators $uC_\varphi$ is to relate operator theoretic properties of $uC_\varphi$ to function theoretic properties of the inducing symbols $\varphi$ and $u$.

This paper is devoted to the study of spectral properties of invertible weighted composition operators acting on the Bloch and Dirichlet spaces, defined in the preliminaries section below. The main references are the papers \cite{3} by Chalendar, Gallardo-Guti\'{e}rrez and Partington, and \cite{2} by Hyv\"{a}rinen et al. In \cite{2} the spectrum of weighted composition operators with automorphic symbols was extensively studied on spaces of analytic functions satisfying certain general conditions introduced in \cite[Section 2.2]{2}. This class contains for example the weighted Bergman spaces and Hardy spaces. However, the Bloch and Dirichlet spaces are not in this class since the bounded analytic functions are not contained in the multipliers of these spaces. A new approach is thus needed, and found partly in \cite[Section 3]{3}, where the spectra of invertible weighted composition operators induced by parabolic and elliptic automorphisms on the Dirichlet space are completely described and the hyperbolic case is left as an open problem.

The paper is organized as follows. In Section \ref{28} we study the multipliers of the Bloch
space, and obtain results similar to those in \cite[Section 2]{3}. Section \ref{29} is devoted to the spectral theory of invertible weighted composition operators $uC_\varphi$ acting on the Bloch space. In particular, we give a description of the spectrum when $\varphi$ is a parabolic or elliptic automorphism of $\mathbb{D}$. In the case of hyperbolic $\varphi$, the spectral radius is computed and we obtain an inclusion of the spectrum in an annulus. Finally, in Section \ref{30} we improve the recent estimates \cite[Theorem 3.3]{3} by Chalendar, Gallardo-Guti\'{e}rrez and Partington of the spectrum of an invertible weighted composition operator on the Dirichlet space when $\varphi$ is a hyperbolic automorphism of $\mathbb{D}$.

\section{Preliminaries}

We begin by recalling some Banach spaces of analytic functions on the unit disc $\mathbb{D}$. The Bloch space $\mathcal{B}$ is the set of functions $f\in H(\mathbb{D})$ such that $\sup_{z \in\mathbb D} \big(1 - |z|^2\big)|f'(z)| < +\infty$, and is equipped with the norm
\[
||f||_{\mathcal B} \ = \ |f(0)| + \sup_{z \in\mathbb D} \big(1 - |z|^2\big)|f'(z)|,\quad f\in\mathcal{B}.
\]
The Dirichlet space $\mathcal D$ consists of functions $f \in H(\mathbb{D})$ such that $f^{\prime} \in A^2$, where
the Bergman space $A^2$ is the set of analytic functions on $\mathbb D$ such that
$$||f||^2_{A^2} \ = \ \int_{\mathbb D} |f(z)|^2  dA(z) < +\infty,$$
with normalized Lebesgue measure $dA(\cdot)$ on $\mathbb{D}$. The Dirichlet norm is defined as
\[
||f||^2 _{\mathcal D} \ = \ |f(0)|^2 + \int_{\mathbb D} |f'(z)|^2 dA(z),\quad f\in\mathcal{D}.
\]
Other spaces of analytic functions on the unit disc $\mathbb{D}$ used in this paper are the disc algebra $A(\mathbb{D})$, consisting of functions continuous on the closed unit disc, the weighted Banach
spaces of analytic functions 
$$H_{v_s}^{\infty} \ = \ \bigg\{f \in H(\mathbb D): ||f||_{H^\infty_{v_s}} \, = \, \sup_{z \in \mathbb D} v_s(z) |f(z)| < +\infty \bigg\},$$ 
where $0 < s < \infty$ and $v_s(z) = (1 - |z|^2)^s$ is the standard weight, and the space $H^{\infty}$ of bounded analytic functions on $\mathbb{D}$ with supremum norm $\|\cdot\|_\infty$. The spectrum and spectral radius of an operator $T: \mathcal{X} \rightarrow \mathcal{X}$ on a space $\mathcal{X}$  are denoted respectively 
by $\sigma_{\mathcal X}(T)$ and $r_{\mathcal X}(T)$.  A good reference for operator theory in function spaces is the monograph \cite{13} by Zhu.

When dealing with composition operators it is customary to denote the $n$-th iterate of a selfmap $\varphi$ of $\mathbb{D}$ by $\varphi_n$, that is
\begin{equation*}
\varphi_n := \underbrace{\varphi \circ \varphi \circ \dots \circ \varphi}_{n \text{ times}}
\end{equation*}
with $\varphi_{0}$ representing the identity map, and it is easy to check that 
\begin{equation*}
(uC_\varphi)^n f(z)=u(z)u(\varphi(z))\cdots u(\varphi_{n-1}(z))
f(\varphi_n(z)),\quad f\in H(\mathbb{D}),\ z\in\mathbb{D}.
\end{equation*}
This can also be stated as $(u C_\varphi)^n = u_{(n)} C_{\varphi_n}$, where $u_{(0)}:=1$ and
\begin{equation*}
u_{(n)} := \prod_{j=0}^{n-1} u\circ \varphi_j\in H(\mathbb{D}), \quad n \in \mathbb{N}. 
\end{equation*}
It turns out that the spectral analysis of invertible weighted composition operators
$uC_\varphi$ strongly depends on the type of the (necessarily) automorphic symbol
$\varphi$. Recall that a nontrivial automorphism $\varphi$ of $\mathbb{D}$ is called \emph{elliptic} if it has a unique fixed point in $\mathbb{D}$, \emph{parabolic} if $\varphi$ has Denjoy--Wolf fixed point $a$ in
$\partial\mathbb{D}$ with $\varphi'(a)=1$, and \emph{hyperbolic} if $\varphi$ has Denjoy--Wolf fixed point $a\in\partial\mathbb{D}$ (the so-called
\emph{attractive fixed point}) with $0 < \varphi'(a) < 1$ and a
\emph{repulsive fixed point} $b\in\partial\mathbb{D}$ with
$\varphi'(b)=1/\varphi'(a)$ (see \cite[Section 2.3.]{4}). When computing the spectrum, we will make use of the formula
\begin{equation}\label{11}
  \lim_{n \to  \infty}{\big(1-|\varphi_{n}(0)|\big)^{\frac{1}{n}}}  =  \varphi^{\prime}(a),
\end{equation}
which is valid for parabolic and hyperbolic automorphisms $\varphi$ of $\mathbb{D}$ with Denjoy--Wolf point $a$, see~\cite[pp.~251--252]{4}.

\section{Multiplier spaces}\label{28}

In this section we consider the multiplier spaces 
\begin{equation*}
\mathcal{M}(\mathcal{X}) := \{u \in H(\mathbb{D}) : M_u : \mathcal{X} \rightarrow \mathcal{X} \text{ is bounded}\} \subset \mathcal{X}
\end{equation*}
and
\begin{equation*}
\mathcal{M}(\mathcal{X},\varphi) := \{u \in H(\mathbb{D}) : uC_{\varphi} : \mathcal{X} \rightarrow \mathcal{X} \text{ is bounded}\} \subset \mathcal{X},
\end{equation*}
where $\mathcal{X}$ is either the Bloch space $\mathcal{B}$ or the Dirichlet space $\mathcal{D}$, and the multiplication operator is defined in the obvious way $M_u f = uf$. The main results of this section are Theorem \ref{31}, where we characterize those $\varphi$ for which $\mathcal{M}(\mathcal{B},\varphi) = \mathcal{B}$, and Theorem~\ref{2} where we show that $\mathcal{M}(\mathcal{B},\varphi) = \mathcal{M}(\mathcal{B})$ whenever $\varphi$ is a finite Blaschke product. The Dirichlet space versions of the mentioned results are given in \cite[Theorems 2.2--2.3]{3}. 

The multiplier space $\mathcal{M}(\mathcal{D})$ was characterized by Stegenga in \cite{12} as the set of functions $u \in H^{\infty}$ such that the multiplication operator $M_{u^{\prime}} : \mathcal{D} \rightarrow A^2$ is bounded ($|u^{\prime}(z)|^2dA(z)$ being a Carleson measure for $\mathcal{D}$). For the Bloch space it is known from \cite{6} that a weighted composition operator $uC_{\varphi} : \mathcal{B} \rightarrow \mathcal{B}$ is bounded if and only if the following two conditions hold: 
\begin{equation}\label{24}
\sup_{z \in \mathbb{D}}\,\big(1-|z|^2\big)|u^{\prime}(z)|\log\frac{e}{1-|\varphi(z)|^2} \ < \ +\infty
\end{equation}
\begin{equation}\label{25}
\sup_{z \in \mathbb{D}}\,\frac{1-|z|^2}{1-|\varphi(z)|^2}|u(z)\varphi^{\prime}(z)| \ < \ +\infty.
\end{equation}
From this follows that $u \in \mathcal{M}(\mathcal{B})$ if and only if
\begin{equation}\label{26}
\sup_{z \in \mathbb{D}}\,\big(1-|z|^2\big)|u^{\prime}(z)|\log\frac{e}{1-|z|^2} \ < \ +\infty 
\end{equation}
\begin{equation}\label{27}
u \in H^{\infty}.
\end{equation}
Condition (\ref{26}) can also be related to a multiplication operator in a similar fashion as in the Dirichlet case:
\begin{lem}\label{20}
If the function $u : \mathbb{D} \rightarrow \mathbb{C}$ is analytic, then the multiplication operator $M_{u^{\prime}} : \mathcal{B} \rightarrow H_{v_1}^{\infty}$ is bounded if and only if condition \textnormal{(\ref{26})} holds. 
\end{lem}

\begin{proof}
If condition \textnormal{(\ref{26})} holds then
\begin{align*}
\|M_{u^{\prime}}\|_{\mathcal{B} \rightarrow H_{v_1}^{\infty}} \ &= \ \sup_{\|f\|_{\mathcal{B}}=1}\|M_{u^{\prime}}f\|_{H_{v_1}^{\infty}} \ = \ \sup_{\|f\|_{\mathcal{B}}=1}\sup_{z \in \mathbb{D}}\,\big(1-|z|^2\big)|u^{\prime}(z)f(z)| \\
&\leq \ \sup_{\|f\|_{\mathcal{B}}=1}\sup_{z \in \mathbb{D}}\,\big(1-|z|^2\big)\big|u^{\prime}(z)\big|\alpha\|f\|_{\mathcal{B}}\log\frac{e}{1-|z|^2} \\
&= \ \alpha\cdot\sup_{z \in \mathbb{D}}\,\big(1-|z|^2\big)|u^{\prime}(z)|\log\frac{e}{1-|z|^2} \ <  +\infty,  
\end{align*}
where we used that every Bloch function $f$ satisf{}ies
\begin{equation}\label{8}
\sup_{z \in \mathbb{D}}\,\frac{|f(z)|}{\log\frac{e}{1-|z|^2}} \ \leq \ \alpha\|f\|_{\mathcal{B}}
\end{equation}
for some positive constant $\alpha$ independent of $f$. On the other hand, if the operator $M_{u^{\prime}} : \mathcal{B} \rightarrow H_{v_1}^{\infty}$ is bounded then there is a constant $c > 0$ such that for every $f \in \mathcal{B}$
\begin{equation*}
\|M_{u^{\prime}}(f)\|_{H_{v_1}^{\infty}} \ \leq \ c\|f\|_{\mathcal{B}}.
\end{equation*} 
When this is applied to the Bloch functions $f_a(z):= \log\frac{e}{1-\bar{a}z}$ for $a \in \mathbb{D}$, we obtain
\begin{equation*}
\big(1-|z|^2\big)|u^{\prime}(z)|\Big|\log\frac{e}{1-\bar{a}z}\Big| \ \leq \ \|M_{u^{\prime}}(f_a)\|_{H_{v_1}^{\infty}} \ \leq \ c\|f_a\|_{\mathcal{B}} \ \leq \ 2c 
\end{equation*} 
for every $z,a \in \mathbb{D}$, since $\|f_a\|_{\mathcal{B}}  \leq  2$ for every $a \in \mathbb{D}$. Now choose $a = z$ and take supremum over $z \in \mathbb{D}$ to get (\ref{26}).
\end{proof}

\begin{thm}\label{31}
Let $\varphi$ be an analytic selfmap of $\mathbb{D}$. Then $\mathcal{M}(\mathcal{B},\varphi) = \mathcal{B}$ if and only if
\begin{enumerate}
\item $\|\varphi\|_{\infty} < 1$ and
\item $\varphi \in \mathcal{M}(\mathcal{B})$.
\end{enumerate}
\end{thm}

\begin{proof}
Assume first that (1) and (2) hold and choose $u \in \mathcal{B}$. Then
\begin{equation*}
\sup_{z \in \mathbb{D}}\,(1-|z|^2)|u^{\prime}(z)|\log\frac{e}{1-|\varphi(z)|^2} \ \leq \ \|u\|_{\mathcal{B}}\log\frac{e}{1-\|\varphi\|_{\infty}^2} \ < \ +\infty,
\end{equation*}
so (\ref{24}) holds. We also have that 
\begin{align*}
&\ \ \ \ \ \ \ \ \ \ \ \ \ \ \ \ \ \ \ \ \ \ \ \ \ \sup_{z \in \mathbb{D}}\,\frac{1-|z|^2}{1-|\varphi(z)|^2}|u(z)\varphi^{\prime}(z)| \\
&\leq \ \frac{1}{1-\|\varphi\|_{\infty}^2}\bigg(\sup_{z \in \mathbb{D}}\,(1-|z|^2)|\varphi^{\prime}(z)|\log\frac{e}{1-|z|^2}\bigg)\cdot\bigg(\sup_{z \in \mathbb{D}}\,\frac{|u(z)|}{\log\frac{e}{1-|z|^2}}\bigg),
\end{align*}
which is finite because $u$ is a Bloch function (see (\ref{8})) and $\varphi \in \mathcal{M}(\mathcal{B})$ (replace $u$ with $\varphi$ in (\ref{26})). Thus (\ref{25}) also holds, and so $u \in \mathcal{M}(\mathcal{B},\varphi)$, which shows that $\mathcal{M}(\mathcal{B},\varphi) = \mathcal{B}$ (the inclusion $\mathcal{M}(\mathcal{B},\varphi) \subset \mathcal{B}$ being trivial). 

Assume on the other hand that $\mathcal{M}(\mathcal{B},\varphi) = \mathcal{B}$, so that $uC_{\varphi}: \mathcal{B} \rightarrow \mathcal{B}$ is bounded for every $u \in \mathcal{B}$. This assumption ensures that if $f \in \mathcal{B}$ then $(f\circ\varphi) \cdot u \in \mathcal{B}$ for every $u \in \mathcal{B}$, which means that the multiplication operator $M_{f\circ \varphi}: \mathcal{B} \rightarrow \mathcal{B}$ is well defined and hence bounded by the Closed Graph Theorem. Thus $C_{\varphi}f = f\circ\varphi \in \mathcal{M}(\mathcal{B}) \subset H^{\infty}$, so that the composition operator $C_{\varphi}: \mathcal{B} \rightarrow H^{\infty}$ is well defined and bounded. Assume to reach a contradiction that $\|\varphi\|_{\infty} = 1$. Choose a function $f \in \mathcal{B}\setminus H^{\infty}$ such that $\|f\|_{\mathcal{B}}=1$. Since $f$ is unbounded there is a sequence $\{\omega_n\}_{n=1}^\infty\subset\mathbb{D}$, $|\omega_n|\rightarrow 1^-$, such that $|f(\omega_n)|> n$ for every $n\in \mathbb{N}$. Since $\|\varphi\|_{\infty} = 1$, there is a sequence $\{z_n\}_{n=1}^\infty\subset\mathbb{D}$ such that $|\varphi(z_n)| = |\omega_n|$ for $n$ large enough, say $n \geq n_0$. Now choose the sequence $\{\theta_n\}_{n=n_0}^\infty$ so that $\omega_n=e^{i\theta_n}\varphi(z_n)$ and define $f_n(z):= f(e^{i\theta_n}z)$. Then $\|f_n\|_{\mathcal{B}} = \|f\|_{\mathcal{B}} = 1$ for every $n$, but $|C_{\varphi}f_n(z_n)| = |f_n(\varphi(z_n))| = |f(\omega_n)| > n$ so that $\|C_{\varphi}f_n\|_{\infty} > n$. This contradicts the boundedness of $C_{\varphi}:\mathcal{B}\rightarrow H^{\infty}$, so we must have $\|\varphi\|_{\infty} < 1$. 

It remains to show that $\varphi \in \mathcal{M}(\mathcal{B})$. The boundedness of $uC_{\varphi}: \mathcal{B} \rightarrow \mathcal{B}$ for every $u \in \mathcal{B}$ and (\ref{25}) imply that
\begin{equation*}
\sup_{z \in \mathbb{D}}\,(1-|z|^2)|u(z)\varphi^{\prime}(z)| \ < \ +\infty
\end{equation*}
for every $u \in \mathcal{B}$. The operator $M_{\varphi^{\prime}} : \mathcal{B} \rightarrow H_{v_1}^{\infty}$ is hence well defined and bounded. According to Lemma \ref{20} this is equivalent to
\begin{equation*}
\sup_{z \in \mathbb{D}}\,(1-|z|^2)|\varphi^{\prime}(z)|\log\frac{e}{1-|z|^2} \ < \ +\infty,
\end{equation*} 
so $\varphi \in \mathcal{M}(\mathcal{B})$ and the proof is complete.
\end{proof}

\begin{thm}\label{2}
Assume that $\varphi$ is a finite Blaschke product. Then $\mathcal{M}(\mathcal{B},\varphi) = \mathcal{M}(\mathcal{B})$.
\end{thm}

\begin{proof}
If $u \in \mathcal{M}(\mathcal{B})$ then $M_u : \mathcal{B} \rightarrow \mathcal{B}$ is bounded and hence $uC_{\varphi} = M_uC_{\varphi} : \mathcal{B} \rightarrow \mathcal{B}$ is bounded, so that $u \in \mathcal{M}(\mathcal{B},\varphi)$ and $\mathcal{M}(\mathcal{B}) \subset \mathcal{M}(\mathcal{B},\varphi)$. Assume on the other hand that $u \in \mathcal{M}(\mathcal{B},\varphi)$, so that $uC_{\varphi} : \mathcal{B} \rightarrow \mathcal{B}$ is bounded and $u \in \mathcal{B}$. Since $\varphi$ is a finite Blaschke product we get from Lemma 1 in \cite{5} that
\begin{equation*}
\frac{1-|\varphi(z)|^2}{1-|z|^2} \ \leq \ \sum_{j=1}^{N}{\frac{1+|\varphi_j(0)|}{1-|\varphi_j(0)|}} =: K,
\end{equation*} 
where $N$ is the degree of the Blaschke product. Hence
\begin{align*}
&\ \ \ \ \ \ \ \ \ \ \ \ \ \ \ \ \ \ \sup_{z \in \mathbb{D}}\,(1-|z|^2)|u^{\prime}(z)|\log\frac{e}{1-|z|^2} \\
&\leq \ \sup_{z \in \mathbb{D}}\,(1-|z|^2)|u^{\prime}(z)|\log\frac{e}{1-|\varphi(z)|^2} \ + \ \log K \cdot \|u\|_{\mathcal{B}} \ < \ +\infty,
\end{align*}
which is (\ref{26}). It remains to show that $u \in H^{\infty}$, so let $S$ be the finite supremum in (\ref{25}). For every $z \in \mathbb{D}$ for which the derivative of $\varphi$ is nonzero we have that
\begin{equation*}
|u(z)| \ \leq \ S\cdot\frac{1-|\varphi(z)|^2}{1-|z|^2}\cdot\frac{1}{|\varphi^{\prime}(z)|} \ \leq \ SK \cdot\frac{1}{|\varphi^{\prime}(z)|}.
\end{equation*} 
Since $\varphi$ is a finite Blaschke product, the derivative $\varphi^{\prime}$ is analytic in an open neighborhood of the closed unit disc, is nonzero on the unit circle $|z|=1$ and has only a finite number of zeros elsewhere. This implies that we can choose $\delta > 0$ so that $m_{\delta} := \min_{\delta \leq |z| \leq 1}\,|\varphi^{\prime}(z)| > 0$, which shows that $u$ is bounded by $SK\frac{1}{m_{\delta}}$ on $\delta \leq |z| < 1$ and hence bounded on $\mathbb{D}$. Thus $u \in \mathcal{M}(\mathcal{B})$, so that $\mathcal{M}(\mathcal{B},\varphi) \subset \mathcal{M}(\mathcal{B})$, and hence $\mathcal{M}(\mathcal{B},\varphi) = \mathcal{M}(\mathcal{B})$.    
\end{proof}

In the next sections we will study the spectrum of \emph{invertible} weighted composition operators on the Bloch and Dirichlet spaces. The invertibility of weighted composition operators has been characterized on various spaces of functions by many authors, see for example \cite{1,8,2}. The following result is a consequence of \cite[Corollary
2.3]{1}.

\begin{thm}\label{3}
Assume that $\mathcal{X}$ is either the Bloch space $\mathcal{B}$ or the Dirichlet space $\mathcal{D}$, and let $uC_{\varphi} : \mathcal{X} \rightarrow \mathcal{X}$ be a bounded weighted composition operator on $\mathcal{X}$. Then $uC_{\varphi}$ is invertible on $\mathcal{X}$ if and only if $u \in \mathcal{M}(\mathcal{X})$, $u$ is bounded away from zero on $\mathbb{D}$ and $\varphi$ is an automorphism of $\mathbb{D}$. In such a case the inverse operator of $uC_{\varphi}: \mathcal{X} \rightarrow \mathcal{X}$ is also a weighted composition operator, given by
\begin{equation*}
\big(uC_{\varphi}\big)^{-1} \ = \ \frac{1}{u\circ\varphi^{-1}}C_{\varphi^{-1}}.
\end{equation*}
\end{thm}

\section{Spectra on the Bloch space}\label{29}

In this section we study the spectrum of invertible weighted composition operators on the Bloch space. The investigation is divided into three cases, to cover parabolic, hyperbolic and elliptic automorphisms. Our approach is based on the papers \cite{3} and \cite{2}, but new ideas are still needed. The following lemma will be useful in the sequel.

\begin{lem}\label{9}
If $\varphi$ is an automorphism of $\mathbb{D}$, then $r_{\mathcal{B}}(C_{\varphi}) = 1$.
\end{lem}

\begin{proof}
By \cite[Corollary 2]{7} and \cite[Lemma 6]{10} we have the following estimate of the composition operator norm for any automorphism $\varphi$ of $\mathbb{D}$ and $n \in \mathbb{N}$:
\begin{equation}\label{12}
1 \ \leq \ \|C_{\varphi_n}\|_{\mathcal{B} \rightarrow \mathcal{B}} \ \leq \ 1 + \frac{1}{2}\log\frac{1+|\varphi_n(0)|}{1-|\varphi_n(0)|} \ \leq \ 1 + \varrho(\varphi(0),0)n,
\end{equation}
where 
\begin{equation*}
\varrho(z,w) := \frac{1}{2}\log\frac{1+\Big|\dfrac{z-w}{1-\bar{z}w}\Big|}{1-\Big|\dfrac{z-w}{1-\bar{z}w}\Big|}
\end{equation*}
is the hyperbolic distance on $\mathbb{D}$. If $\varphi(0) = 0$ then obviously $r_{\mathcal{B}}(C_{\varphi}) = 1$. If on the other hand $\varrho(\varphi(0),0) \neq 0$, then we obtain the following estimate from (\ref{12}) for every $n \in \mathbb{N}$:
\begin{align*}
1 \ \leq \ \|C_{\varphi_n}\|_{\mathcal{B} \rightarrow \mathcal{B}}^{\frac{1}{n}} \ \leq \ \left[ \big(1 + \varrho(\varphi(0),0)n\big)^{\frac{1}{\varrho(\varphi(0),0)n}}\right]^{\varrho(\varphi(0),0)},
\end{align*}
and since $\lim_{x \rightarrow +\infty}(1+x)^{\frac{1}{x}} = 1$ we conclude that $r_{\mathcal{B}}(C_{\varphi}) = 1$.  
\end{proof}

\subsection{The parabolic case}

We begin by describing the spectrum of invertible weighted composition operators $uC_{\varphi}$ on the Bloch space induced by parabolic automorphisms $\varphi$, and for the proof we need the following result.

\begin{lem}\label{21}
Suppose that $\varphi$ is a parabolic automorphism of $\mathbb{D}$ with unique fixed point $a \in \partial \mathbb{D}$, and assume that $u \in A(\mathbb{D})$ is bounded away from zero on $\mathbb{D}$. Then
\begin{equation*}
\lim_{n \rightarrow \infty}\|u_{(n)}\|_{\infty}^{\frac{1}{n}} = |u(a)|.
\end{equation*}
\end{lem}

\begin{proof}
See the proof of \cite[Lemma 4.2]{2}. 
\end{proof}

\begin{thm}\label{13}
Suppose that the weighted composition operator $uC_{\varphi} : \mathcal{B} \rightarrow \mathcal{B}$ is invertible on the Bloch space and assume that the automorphism $\varphi$ is parabolic, with unique fixed point $a \in \partial \mathbb{D}$. If $u \in A(\mathbb{D})$, then
\begin{equation*}
\sigma_{\mathcal{B}}(uC_{\varphi}) \ = \ \big\{\lambda \in \mathbb{C} : |\lambda| = |u(a)|\big\}.
\end{equation*} 
\end{thm}

\begin{proof}
We begin by showing that the spectrum is contained in the proposed circle of radius $|u(a)|$. According to Theorem \ref{3}, $u$ belongs to $\mathcal{M}(\mathcal{B})$ and is bounded away from zero on $\mathbb{D}$, from which follows that $u(a) \neq 0$. Since $r_{\mathcal{B}}(C_{\varphi}) = 1$ by Lemma \ref{9} and $(uC_{\varphi})^n = u_{(n)} C_{\varphi_n} = M_{u_{(n)}}C_{\varphi_n}$ by Theorem \ref{2}, we only need to focus on the operator norm of $M_{u_{(n)}} : \mathcal{B} \rightarrow \mathcal{B}$:
\begin{align*}
& \ \ \ \ \ \ \ \ \ \ \ \ \ \ \ \ \ \ \ \ \ \  \ \ \ \ \  \|M_{u_{(n)}}\|_{\mathcal{B} \rightarrow \mathcal{B}} \ = \ \sup_{\|f\|_{\mathcal{B}} \leq 1}{\|u_{(n)}\cdot f\|_{\mathcal{B}}} \\
&\leq \ \sup_{\|f\|_{\mathcal{B}}\leq1}\sup_{z \in \mathbb{D}}{(1-|z|^2)|u_{(n)}^{\prime}(z)f(z)|} \ + \, \sup_{\|f\|_{\mathcal{B}}\leq1}\sup_{z \in \mathbb{D}}{(1-|z|^2)|u_{(n)}(z)f^{\prime}(z)|} \ +  \\
&\ \ \ \ \sup_{\|f\|_{\mathcal{B}}\leq1}|u_{(n)}(0)f(0)| \\
&= \ \sup_{\|f\|_{\mathcal{B}}\leq1}\sup_{z \in \mathbb{D}}{(1-|z|^2)\bigg|\sum_{j=0}^{n-1}{\frac{u_{(n)}(z)}{u \circ \varphi_j(z)}\cdot (u \circ \varphi_j)^{\prime}(z)}\bigg||f(z)|} \ \ + \\ 
&\ \ \ \ \sup_{\|f\|_{\mathcal{B}}\leq1}\sup_{z \in \mathbb{D}}{(1-|z|^2)|u_{(n)}(z)f^{\prime}(z)|} \ + \ \sup_{\|f\|_{\mathcal{B}}\leq1}|u_{(n)}(0)f(0)| \\
&\leq \ \sum_{j=0}^{n-1}{\Big\|\frac{u_{(n)}}{u \circ \varphi_j}\Big\|_{\infty}}\sup_{\|f\|_{\mathcal{B}}\leq1}\sup_{z \in \mathbb{D}}\,(1-|z|^2)|u^{\prime}(\varphi_j(z))||\varphi_j^{\prime}(z)||f(z)| \ + \ 2\|u_{(n)}\|_{\infty} \\
&= \ \sum_{j=0}^{n-1}{\Big\|\frac{u_{(n)}}{u \circ \varphi_j}\Big\|_{\infty}}\sup_{\|f\|_{\mathcal{B}}\leq1}\sup_{z \in \mathbb{D}}\,(1-|\varphi_j(z)|^2)|u^{\prime}(\varphi_j(z))f(z)| \ + \ 2\|u_{(n)}\|_{\infty} \\
&= \ \sum_{j=0}^{n-1}{\Big\|\frac{u_{(n)}}{u \circ \varphi_j}\Big\|_{\infty}}\sup_{\|f\|_{\mathcal{B}}\leq 1}\sup_{z \in \mathbb{D}}\,(1-|z|^2)|u^{\prime}(z)f(\psi_j(z))| \ + \ 2\|u_{(n)}\|_{\infty},
\end{align*}
where we used that $|\varphi^{\prime}(z)| = \frac{1-|\varphi(z)|^2}{1-|z|^2}$ if $\varphi$ is an automorphism of $\mathbb{D}$, and introduced the function $\psi := \varphi^{-1}$ to simplify notation. Note that $\psi$ is also a parabolic automorphism with fixed point $a$, and $\psi_j = (\varphi^{-1})_j = \varphi_j^{-1}$. Furthermore,
\begin{align*}
\sup_{z \in \mathbb{D}}\,(1-|z|^2)|u^{\prime}(z)f(\psi_j(z))| \ &= \ \|M_{u^{\prime}}(f \circ \psi_j)\|_{H_{v_1}^{\infty}} \\
&\leq \ \|M_{u^{\prime}}\|_{\mathcal{B} \rightarrow H_{v_1}^{\infty}}\|C_{\psi_j}\|_{\mathcal{B} \rightarrow \mathcal{B}}\|f\|_{\mathcal{B}},
\end{align*}
where the norm $\|M_{u^{\prime}}\|_{\mathcal{B} \rightarrow H_{v_1}^{\infty}}$ is finite since $u \in \mathcal{M}(\mathcal{B})$, as discussed in Lemma \ref{20}. Now using inequality (\ref{12}) and the fact that $u$ is bounded away from zero on $\mathbb{D}$, we obtain the following estimate of the operator norm of $M_{u_{(n)}} : \mathcal{B} \rightarrow \mathcal{B}$: 
\begin{align*}
\|M_{u_{(n)}}\|_{\mathcal{B} \rightarrow \mathcal{B}} \ &\leq \ \sum_{j=0}^{n-1}{\Big\|\frac{u_{(n)}}{u \circ \varphi_j}\Big\|_{\infty}}\|M_{u^{\prime}}\|_{\mathcal{B} \rightarrow H_{v_1}^{\infty}}\|C_{\psi_j}\|_{\mathcal{B} \rightarrow \mathcal{B}} \ + \ 2\|u_{(n)}\|_{\infty} \\ 
&\leq \ \|M_{u^{\prime}}\|_{\mathcal{B} \rightarrow H_{v_1}^{\infty}}\sum_{j=0}^{n-1}{\Big\|\frac{u_{(n)}}{u \circ \varphi_j}\Big\|_{\infty}}\big(1+\varrho(\psi(0),0)j\big) \ + \ 2\|u_{(n)}\|_{\infty} \\
&\leq \ \left[\|M_{u^{\prime}}\|_{\mathcal{B} \rightarrow H_{v_1}^{\infty}}\Big\|\frac{1}{u}\Big\|_{\infty} + 2\right]n\big(1+\varrho(\psi(0),0)n\big)\|u_{(n)}\|_{\infty}.
\end{align*}
Applying this to the spectral radius and using Lemmas \ref{9} and \ref{21} gives
\begin{align*}
r_{\mathcal{B}}(uC_{\varphi}) \ &= \ \lim_{n \rightarrow \infty}\|(uC_{\varphi})^n\|_{\mathcal{B} \rightarrow \mathcal{B}}^{\frac{1}{n}} \ \leq \ \limsup_{n \rightarrow \infty}\|M_{u_{(n)}}\|_{\mathcal{B} \rightarrow \mathcal{B}}^{\frac{1}{n}}\hspace*{0.3mm}r_{\mathcal{B}}(C_{\varphi}) \\
&\leq \ \lim_{n \rightarrow \infty}\left[\|M_{u^{\prime}}\|_{\mathcal{B} \rightarrow H_{v_1}^{\infty}}\Big\|\frac{1}{u}\Big\|_{\infty} + 2\right]^{\frac{1}{n}}n^{\frac{1}{n}}\big(1+\varrho(\psi(0),0)n\big)^{\frac{1}{n}}\|u_{(n)}\|_{\infty}^{\frac{1}{n}} \\
&= \ |u(a)|.
\end{align*}
Since $\big(uC_{\varphi}\big)^{-1}  = \frac{1}{u\circ\varphi^{-1}}C_{\varphi^{-1}}$, by Theorem \ref{3}, where $\varphi^{-1}$ is a parabolic automorphism with unique fixed point $a$, the above result also shows that 
\begin{equation*}
r_{\mathcal{B}}\big((uC_{\varphi})^{-1}\big) \ \leq \ \bigg|\frac{1}{u(\varphi^{-1}(a))}\bigg| \ =  \ |u(a)|^{-1}.
\end{equation*} 
Now if $\lambda \in \sigma_{\mathcal{B}}(uC_{\varphi})$, then we also have that $\lambda^{-1} \in \sigma_{\mathcal{B}}\big((uC_{\varphi})^{-1}\big)$. From this follows that $|\lambda| \leq r_{\mathcal{B}}(uC_{\varphi}) \leq |u(a)|$ and $|\lambda|^{-1} \leq r_{\mathcal{B}}\big((uC_{\varphi})^{-1}\big) \leq |u(a)|^{-1}$, so that $|\lambda| = |u(a)|$. Thus
\begin{equation}\label{7}
\sigma_{\mathcal{B}}(uC_{\varphi}) \ \subseteq \ \big\{\lambda \in \mathbb{C} : |\lambda| = |u(a)|\big\},
\end{equation}  
and obviously $r_{\mathcal{B}}(uC_{\varphi}) = |u(a)|$. 

In order to prove the reverse inclusion in (\ref{7}), let $\lambda$ be a complex number of modulus $|\lambda| = |u(a)| = r_{\mathcal{B}}(uC_{\varphi})$. As in the proof of \cite[Theorem 4.3]{2} it is then, by the Spectral Mapping Theorem, enough to show that
\begin{equation*}
r_{\mathcal{B}}(\lambda - uC_{\varphi}) \ \geq \ 2r_{\mathcal{B}}(uC_{\varphi}),
\end{equation*}
which is done as follows. The sequence $\{z_n\}_{n=0}^{\infty}$, defined by $z_n = \varphi_n(0)$, is interpolating for $H^{\infty}$ since $\varphi$ is a parabolic automorphism (see the comment preceding \cite[Theorem 4.3]{2}), so by the Open Mapping Theorem there is a constant $c > 0$ and a sequence $\{f_n\}_{n=0}^{\infty} \subset H^{\infty}$ such that for all $n \in \mathbb{N}$ we have $\|f_n\|_{\infty} \leq c$ and
\begin{equation}\label{35}
f_n\big(\varphi_k(z_n)\big) \ = \begin{cases}
1, \ \ \text{if} \ k = n \\
0, \ \ \text{if} \ k \neq n. \\
\end{cases}
\end{equation}
Note that $\mathcal{B} \subset H_{v_s}^{\infty}$ for every $s > 0$, since if $f \in \mathcal{B}$ then by (\ref{8})
\begin{equation*}
\|f\|_{H_{v_s}^{\infty}} \ \leq \ \alpha\|f\|_{\mathcal{B}}\sup_{z \in \mathbb{D}}\,(1-|z|^2)^s\log\frac{e}{1-|z|^2} \ < +\infty.
\end{equation*}
Choose some $s > 0$ and let
\begin{equation*}
c_s  :=  \alpha\sup_{z \in \mathbb{D}}\,(1-|z|^2)^s\log\frac{e}{1-|z|^2},
\end{equation*}
so that $\|f\|_{H_{v_s}^{\infty}} \leq c_s\|f\|_{\mathcal{B}}$ for every $f \in \mathcal{B}$. The interpolating sequence $\{f_n\}_{n=0}^{\infty}$ satisfies $\|f_n\|_{\mathcal{B}} \leq \|f_n\|_{\infty} \leq c$ for all $n \in \mathbb{N}$, which gives that
\begin{align*}
\|(\lambda - uC_{\varphi})^{2n}\|_{\mathcal{B} \rightarrow \mathcal{B}} \ &\geq \ c^{-1}\|(\lambda - uC_{\varphi})^{2n}f_n\|_{\mathcal{B}} \\
&\geq \ (c c_s)^{-1}\|(\lambda - uC_{\varphi})^{2n}f_n\|_{H_{v_s}^{\infty}} \\
&\geq \ (c c_s)^{-1}(1-|z_n|^2)^s\big|\left[(\lambda - uC_{\varphi})^{2n}f_n\right](z_n)\big|
\end{align*}
for all $n \in \mathbb{N}$. Furthermore,
\begin{align*}
\left[(\lambda - uC_{\varphi})^{2n}f_n\right](z_n) \ &= \ \sum_{k=0}^{2n}{\binom{2n}{k}\lambda^{2n-k}\left[(-uC_ {\varphi})^kf_n\right](z_n)} \\
&= \ \sum_{k=0}^{2n}{\binom{2n}{k}\lambda^{2n-k}(-1)^ku_{(k)}(z_n)f_n\big(\varphi_k(z_n)\big)} \\
&= \ \binom{2n}{n}(-1)^n\lambda^{n}u_{(n)}(z_n)
\end{align*}
by (\ref{35}), so
\begin{align*}
&\|(\lambda - uC_{\varphi})^{2n}\|_{\mathcal{B} \rightarrow \mathcal{B}} \ \geq \ (c c_s)^{-1} \binom{2n}{n}|\lambda|^{n}|u_{(n)}(z_n)|(1-|z_n|^2)^s \\
&= \ (c c_s)^{-1} \binom{2n}{n}|\lambda|^{n}|u_{(n)}(z_n)|\bigg(\frac{1-|z_n|^2}{1-|\varphi_n(z_n)|^2}\bigg)^s\big(1-|\varphi_n(z_n)|^2\big)^s \\
&= \ (c c_s)^{-1} \binom{2n}{n}|\lambda|^{n}\bigg|\frac{u_{(n)}(z_n)}{\varphi_n^{\prime}(z_n)^s}\bigg|\big(1-|\varphi_{2n}(0)|^2\big)^s \\
&\geq \ (c c_s)^{-1} \binom{2n}{n}|\lambda|^{n}|\omega_{(n)}(z_n)|\big(1-|\varphi_{2n}(0)|\big)^s,
\end{align*}
where the function 
\begin{equation*}
\omega(z):=\frac{u(z)}{\varphi^{\prime}(z)^s}
\end{equation*}
introduced in the proof of \cite[Theorem 4.3]{2} satisfies 
\begin{equation*}
\omega_{(n)}(z) = \frac{u_{(n)}(z)}{\varphi_n^{\prime}(z)^s}.
\end{equation*}
In the same proof it was also shown that
\begin{equation*}
\lim_{n \rightarrow \infty}{|\omega_{(n)}(z_n)|^{\frac{1}{2n}}}  \ = \ \dfrac{|u(a)|^{\frac{1}{2}}}{\varphi^{\prime}(a)^{\frac{s}{2}}} \ = \ |u(a)|^\frac{1}{2},
\end{equation*}
and mentioned that $\lim_{n \rightarrow \infty}{\binom{2n}{n}^{\frac{1}{2n}}} = 2$. Using the parabolic version of the limit (\ref{11}), we see that
\begin{align*}
r_{\mathcal{B}}(\lambda - uC_{\varphi}) \ &= \ \lim_{n \rightarrow \infty}\|(\lambda - uC_{\varphi})^{2n}\|_{\mathcal{B} \rightarrow \mathcal{B}}^{\frac{1}{2n}} \\
&\geq \ \lim_{n \rightarrow \infty}(c c_s)^{-\frac{1}{2n}} \binom{2n}{n}^{\frac{1}{2n}}|\lambda|^{\frac{1}{2}}|\omega_{(n)}(z_n)|^{\frac{1}{2n}}\big(1-|\varphi_{2n}(0)|\big)^{\frac{s}{2n}}\\
&= \ 2|\lambda|^{\frac{1}{2}}|u(a)|^{\frac{1}{2}} \ = \ 2|u(a)| \ = \ 2r_{\mathcal{B}}(uC_{\varphi}).
\end{align*}
Since $r_{\mathcal{B}}(\lambda - uC_{\varphi}) \geq 2r_{\mathcal{B}}(uC_{\varphi})$, we get from the proof of \cite[Theorem 4.3]{2} that $-\lambda \in \sigma_{\mathcal{B}}(uC_{\varphi})$, which shows that
\begin{equation*}
\sigma_{\mathcal{B}}(uC_{\varphi}) \ \supseteq \ \big\{-\lambda \in \mathbb{C} : |\lambda| = |u(a)|\big\} \ = \ \big\{\lambda \in \mathbb{C} : |\lambda| = |u(a)|\big\},
\end{equation*} 
and the proof is complete.
\end{proof}

\subsection{The hyperbolic case}

In this subsection we investigate the spectrum of weighted composition
operators $uC_\varphi\colon \mathcal{B}\to\mathcal{B}$ generated by hyperbolic symbols $\varphi$. The results in this case are not complete. We obtain the spectral radius $r_{\mathcal{B}}(uC_{\varphi})$ and thereby an inclusion of the spectrum in an annulus, which turns out to coincide with the spectrum under additional assumptions on the function $u$. The main result is given in Theorem \ref{17}.

\begin{lem}\label{22}
Suppose that $\varphi$ is a hyperbolic automorphism of $\mathbb{D}$ with fixed points $a,b \in \partial \mathbb{D}$, and assume that $u \in A(\mathbb{D})$ is bounded away from zero on $\mathbb{D}$. Then
\begin{equation*}
\lim_{n \rightarrow \infty}\|u_{(n)}\|_{\infty}^{\frac{1}{n}} = \max\{|u(a)|,|u(b)|\}.
\end{equation*}
\end{lem}

\begin{proof}
See the proof of \cite[Lemma 4.4]{2}. 
\end{proof}

\begin{thm}\label{17}
Suppose that the weighted composition operator $uC_{\varphi} : \mathcal{B} \rightarrow \mathcal{B}$ is invertible on the Bloch space and assume that the automorphism $\varphi$ is hyperbolic, with attractive fixed point $a \in \partial \mathbb{D}$ and repulsive fixed point $b \in \partial\mathbb{D}$. If $u \in A(\mathbb{D})$, then $r_{\mathcal{B}}(uC_{\varphi}) \ = \ \max\{|u(a)|,|u(b)|\}$ and
\begin{equation*}
\sigma_{\mathcal{B}}(uC_{\varphi}) \ \subseteq \ \big\{\lambda \in \mathbb{C} :  \min\{|u(a)|,|u(b)|\} \leq |\lambda| \leq \max\{|u(a)|,|u(b)|\}\big\}.
\end{equation*} 
\end{thm}

\begin{proof}
As in the proof of Theorem \ref{13}, $u$ belongs to $\mathcal{M}(\mathcal{B})$ and is bounded away from zero on $\mathbb{D}$, so that $u(a), u(b) \neq 0$. Since $r_{\mathcal{B}}(C_{\varphi}) = 1$ by Lemma \ref{9} and $(uC_{\varphi})^n = M_{u_{(n)}}C_{\varphi_n}$, it is again enough to consider the operator norm of $M_{u_{(n)}} : \mathcal{B} \rightarrow \mathcal{B}$. Through identical calculations as in the proof of Theorem \ref{13}, observing that $\psi := \varphi^{-1}$ is also a hyperbolic automorphism, we obtain the estimate
\begin{align*}
\|M_{u_{(n)}}\|_{\mathcal{B} \rightarrow \mathcal{B}} \ \leq \ \left[\|M_{u^{\prime}}\|_{\mathcal{B} \rightarrow H_{v_1}^{\infty}}\Big\|\frac{1}{u}\Big\|_{\infty} + 2\right]n\big(1+\varrho(\psi(0),0)n\big)\|u_{(n)}\|_{\infty},
\end{align*}
and so by using Lemmas \ref{9} and \ref{22} one gets that
\begin{align*}
r_{\mathcal{B}}(uC_{\varphi}) \ &= \ \lim_{n \rightarrow \infty}\|(uC_{\varphi})^n\|_{\mathcal{B} \rightarrow \mathcal{B}}^{\frac{1}{n}} \ \leq \ \limsup_{n \rightarrow \infty}\|M_{u_{(n)}}\|_{\mathcal{B} \rightarrow \mathcal{B}}^{\frac{1}{n}}\hspace*{0.3mm}r_{\mathcal{B}}(C_{\varphi}) \\
&\leq \ \lim_{n \rightarrow \infty}\left[\|M_{u^{\prime}}\|_{\mathcal{B} \rightarrow H_{v_1}^{\infty}}\Big\|\frac{1}{u}\Big\|_{\infty} + 2\right]^{\frac{1}{n}}n^{\frac{1}{n}}\big(1+\varrho(\psi(0),0)n\big)^{\frac{1}{n}}\|u_{(n)}\|_{\infty}^{\frac{1}{n}} \\
&= \ \max\{|u(a)|,|u(b)|\}.
\end{align*}
On the other hand, $\|u_{(n)}\|_{\infty} \leq \|M_{u_{(n)}}\|_{\mathcal{B} \rightarrow \mathcal{B}}$ by \cite[Lemma 1]{9}, so
\begin{align*}
\|u_{(n)}\|_{\infty}^{\frac{1	}{n}} \ &\leq \ \|M_{u_{(n)}}\|_{\mathcal{B} \rightarrow \mathcal{B}}^{\frac{1	}{n}} \ = \ \big\|\big(uC_{\varphi}\big)^n\big(C_{\varphi_n}\big)^{-1}\big\|_{\mathcal{B} \rightarrow \mathcal{B}}^{\frac{1}{n}} \\
&\leq \ \big\|\big(uC_{\varphi}\big)^n\big\|_{\mathcal{B} \rightarrow \mathcal{B}}^{\frac{1}{n}}\big\|\big(C_{\varphi^{-1}}\big)^n\big\|_{\mathcal{B} \rightarrow \mathcal{B}}^{\frac{1}{n}}.
\end{align*}
Letting $n$ tend to infinity and observing that $r_{\mathcal{B}}\big(C_{\varphi^{-1}}\big) = 1$ (by Lemma \ref{9} since $\varphi^{-1} \in \textnormal{Aut}(\mathbb{D})$) we see that
\begin{equation*}
\max\{|u(a)|,|u(b)|\} \ \leq \ r_{\mathcal{B}}(uC_{\varphi}),
\end{equation*}  
and thus
\begin{equation*}
r_{\mathcal{B}}(uC_{\varphi}) \ = \ \max\{|u(a)|,|u(b)|\}.
\end{equation*}
Applying the above result to the inverse operator $\big(uC_{\varphi}\big)^{-1}  = \frac{1}{u\circ\varphi^{-1}}C_{\varphi^{-1}}$, where $\varphi^{-1}$ is a hyperbolic automorphism with attractive fixed point $b$ and repulsive fixed point $a$, we get that
\begin{align*}
r_{\mathcal{B}}\big((uC_{\varphi})^{-1}\big) \ &= \ \max\Big\{\big|u\big(\varphi^{-1}(a)\big)\big|^{-1},\big|u\big(\varphi^{-1}(b)\big)\big|^{-1}\Big\} \\
&= \ \dfrac{1}{\min\big\{|u(a)|,|u(b)|\big\}}.
\end{align*}
Now if $\lambda \in \sigma_{\mathcal{B}}(uC_{\varphi})$, then we also have that $\lambda^{-1} \in \sigma_{\mathcal{B}}\big((uC_{\varphi})^{-1}\big)$, so that
\begin{equation*}
|\lambda| \ \leq \ r_{\mathcal{B}}(uC_{\varphi}) \ = \ \max\big\{|u(a)|,|u(b)|\big\}
\end{equation*}
and
\begin{equation*}
|\lambda|^{-1} \ \leq \ r_{\mathcal{B}}\big((uC_{\varphi})^{-1}\big) \ = \ \dfrac{1}{\min\big\{|u(a)|,|u(b)|\big\}},
\end{equation*}
which shows that the spectrum is contained in the claimed annulus.
\end{proof}

In the case when $|u(a)| = |u(b)|$ we are able to improve the previous theorem, to give a complete description of the spectrum. However, we were not able to compute the spectrum when $|u(a)| \neq |u(b)|$. It seems like one needs to consider the cases $|u(a)| < |u(b)|$ and $|u(a)| > |u(b)|$ separately, see for example \cite[Theorem 4.9]{2}.

\begin{thm}
Suppose that the weighted composition operator $uC_{\varphi} : \mathcal{B} \rightarrow \mathcal{B}$ is invertible on the Bloch space and assume that the automorphism $\varphi$ is hyperbolic, with attractive fixed point $a \in \partial \mathbb{D}$ and repulsive fixed point $b \in \partial\mathbb{D}$. If $u \in A(\mathbb{D})$ and $|u(a)|=|u(b)|$, then
\begin{equation*}
\sigma_{\mathcal{B}}(uC_{\varphi}) \ = \ \big\{\lambda \in \mathbb{C} :  |\lambda| = |u(a)|\big\}.
\end{equation*} 
\end{thm}

\begin{proof}
By Theorem \ref{17} it suffices to prove that
\begin{equation*}
\sigma_{\mathcal{B}}(uC_{\varphi}) \ \supseteq \ \big\{\lambda \in \mathbb{C} :  |\lambda| = |u(a)|\big\},
\end{equation*}
so let $\lambda$ be a complex number of modulus $|\lambda| = |u(a)| = r_{\mathcal{B}}(uC_{\varphi})$. As in the proof of Theorem \ref{13} it is then enough to show that $r_{\mathcal{B}}(\lambda - uC_{\varphi}) \geq 2r_{\mathcal{B}}(uC_{\varphi})$. This can be done exactly as in the previously mentioned proof since the sequence $\{z_n\}_{n=0}^{\infty} \subset \mathbb{D}$, defined by $z_n = \varphi_n(0)$, is interpolating for $H^{\infty}$ by \cite[Theorem 2.65]{4}. Using the same notation and performing the same calculations as in the proof of Theorem \ref{13}, we obtain the estimate
\begin{equation*}
\|(\lambda - uC_{\varphi})^{2n}\|_{\mathcal{B} \rightarrow \mathcal{B}} \ \geq \ (c c_s)^{-1} \binom{2n}{n}|\lambda|^{n}|\omega_{(n)}(z_n)|\big(1-|\varphi_{2n}(0)|\big)^s.
\end{equation*} 
Now since
\begin{equation*}
\lim_{n \rightarrow \infty}{|\omega_{(n)}(z_n)|^{\frac{1}{2n}}} \ = \ \dfrac{|u(a)|^{\frac{1}{2}}}{\varphi^{\prime}(a)^{\frac{s}{2}}},
\end{equation*}
as in \cite[Theorem 4.3]{2}, this gives that
\begin{align*}
r_{\mathcal{B}}(\lambda - uC_{\varphi}) \ &\geq \ \lim_{n \rightarrow \infty}(c c_s)^{-\frac{1}{2n}} \binom{2n}{n}^{\frac{1}{2n}}|\lambda|^{\frac{1}{2}}|\omega_{(n)}(z_n)|^{\frac{1}{2n}}\big(1-|\varphi_{2n}(0)|\big)^{\frac{s}{2n}}\\
&= \ 2|\lambda|^{\frac{1}{2}}\dfrac{|u(a)|^{\frac{1}{2}}}{\varphi^{\prime}(a)^{\frac{s}{2}}}\varphi^{\prime}(a)^{s} \ = \ 2|u(a)|\varphi^{\prime}(a)^{\frac{s}{2}}  \ = \ 2r_{\mathcal{B}}(uC_{\varphi})\varphi^{\prime}(a)^{\frac{s}{2}},
\end{align*}  
where we used the limit (\ref{11}) with $0 < \varphi^{\prime}(a) < 1$. The inequality 
\begin{equation*}
r_{\mathcal{B}}(\lambda - uC_{\varphi}) \ \geq \ 2r_{\mathcal{B}}(uC_{\varphi})\varphi^{\prime}(a)^{\frac{s}{2}}
\end{equation*}
holds for every $s > 0$ (see the latter part of the proof of Theorem \ref{13}) so we may take limits as $s \rightarrow 0$ to obtain $r_{\mathcal{B}}(\lambda - uC_{\varphi}) \geq 2r_{\mathcal{B}}(uC_{\varphi})$.  
\end{proof}

\subsection{The elliptic case}

We now turn to the spectrum of invertible weighted composition operators
$uC_\varphi$ on the Bloch space when $\varphi$ is an elliptic automorphism of
$\mathbb{D}$.  Though the methods of proof here are standard, some
minor modifications are necessary and we thus present them.

\begin{lem}\label{18}
If $u \in \mathcal{M}(\mathcal{B})$ and $\frac{1}{u} \in H^{\infty}$, then $\frac{1}{u} \in \mathcal{M}(\mathcal{B})$.
\end{lem}

\begin{proof}
The function $f(z) := \frac{1}{u(z)}$ is bounded by assumption, so we only need to show that it satisfies condition (\ref{26}):
\begin{align*}
\sup_{z \in \mathbb{D}}\,(1-|z|^2)|f^{\prime}(z)|\log\frac{e}{1-|z|^2} \ &= \ \sup_{z \in \mathbb{D}}\,(1-|z|^2)\bigg|\frac{u^{\prime}(z)}{u(z)^2}\bigg|\log\frac{e}{1-|z|^2} \\
&\leq \ \Big\|\frac{1}{u}\Big\|_{\infty}^2\sup_{z \in \mathbb{D}}\,(1-|z|^2)|u^{\prime}(z)|\log\frac{e}{1-|z|^2},
\end{align*}
which is finite since $u \in \mathcal{M}(\mathcal{B})$.
\end{proof}

\begin{thm}\label{32}
Suppose that the weighted composition operator $uC_{\varphi}: \mathcal{B} \rightarrow \mathcal{B}$ is bounded on the Bloch space and assume that $u \in A(\mathbb{D})$ and that $\varphi$ is an elliptic automorphism, with unique fixed point $a \in \mathbb{D}$. Then 
\begin{enumerate}
\item[1)] either there is a positive integer $j$ such that $\varphi_j(z) = z$ for all $z \in \mathbb{D}$, in which case, if $m$ is the smallest such integer, then
\begin{equation*}
\sigma_{\mathcal{B}}(uC_{\varphi}) \ = \ \overline{\big\{\lambda \in \mathbb{C} : \lambda^m  =  u_{(m)}(z) \text{ for some  } z \in \mathbb{D}\big\}},
\end{equation*} 
\item[2)] or $\varphi_n \neq$ \textnormal{Id} for every $n \in \mathbb{N}$, in which case, if $uC_{\varphi}: \mathcal{B} \rightarrow \mathcal{B}$ is invertible, then
\begin{equation*}
\sigma_{\mathcal{B}}(uC_{\varphi}) \ = \ \big\{\lambda \in \mathbb{C} : |\lambda| = |u(a)|\big\}.
\end{equation*} 
\end{enumerate} 
\end{thm}

\begin{proof}
The proof of 1) is identical to the proof of \cite[Theorem 4.11]{2}, because as noted in \cite[Section 3.1]{3} we can use the result of Lemma \ref{18} to prove that
\begin{equation*}
\sigma_{\mathcal{B}}(uC_{\varphi}) \ \subseteq \ \overline{\big\{\lambda \in \mathbb{C} : \lambda^m  =  u_{(m)}(z) \text{ for some  } z \in \mathbb{D}\big\}}.
\end{equation*}
The proof of 2) goes as in \cite[Theorem 4.14]{2} and it relies on \cite[Lemma 4.13]{2}, which is also true for the Bloch space after a minor modification in the proof. Namely, it is assumed that $\varphi(z) = \mu z$, where $\mu = e^{2\pi\theta i}$ and $\theta$ is irrational, $u \in A(\mathbb{D})$ and $uC_{\varphi}$ is invertible, and proven that $r_{\mathcal{A}}\big(uC_{\varphi}\big) = |u(0)|$ by methods also valid for the Bloch space, except the one showing that
\begin{equation}\label{19}
r_{\mathcal{A}}\big(uC_{\varphi}\big) \ \leq \ \limsup_{n \rightarrow \infty}{\|u_{(n)}\|_{\infty}^{\frac{1}{n}}}.
\end{equation} 
Here $\mathcal{A}$ stands for a space satisfying conditions (C1), (C2) and (C3) as defined in \cite[Section 2.2]{2}. However, by the same calculations used to prove Theorem \ref{13} we get that
\begin{align*}
\|M_{u_{(n)}}\|_{\mathcal{B} \rightarrow \mathcal{B}}^{\frac{1}{n}} \ \leq \ \left[\|M_{u^{\prime}}\|_{\mathcal{B} \rightarrow H_{v_1}^{\infty}}\Big\|\frac{1}{u}\Big\|_{\infty} + 2\right]^{\frac{1}{n}}n^{\frac{1}{n}}\big(1+\varrho(\psi(0),0)n\big)^{\frac{1}{n}}\|u_{(n)}\|_{\infty}^{\frac{1}{n}},
\end{align*}
and from this follows that
\begin{align*}
r_{\mathcal{B}}\big(uC_{\varphi}\big) \ \leq \ r_{\mathcal{B}}\big(C_{\varphi}\big)\limsup_{n \rightarrow \infty}{\|M_{u_{(n)}}\|_{\mathcal{B} \rightarrow \mathcal{B}}^{\frac{1}{n}}} \ \leq \ \limsup_{n \rightarrow \infty}{\|u_{(n)}\|_{\infty}^{\frac{1}{n}}},
\end{align*}
where we used that $r_{\mathcal{B}}\big(C_{\varphi}\big) = 1$ by Lemma \ref{9}. Thus (\ref{19}) also holds for the Bloch space, and we can use the proof of \cite[Theorem 4.14]{2} to obtain 2). 
\end{proof}

\section{Spectra on the Dirichlet space}\label{30}

The spectra of invertible weighted composition operators induced by para-bolic and elliptic automorphisms on the Dirichlet space were completely described in \cite{3}. In the hyperbolic case and under essentially the same assumptions as in Theorem \ref{33} below, it was also shown in \cite[Theorem 3.3]{3} that $r_{\mathcal{D}}(uC_{\varphi}) \leq \max\{|u(a)|,|u(b)|\}\frac{1}{\mu}$ and
\[
\sigma_{\mathcal{D}}(uC_{\varphi}) \subseteq \bigl\{\lambda \in
\mathbb{C} : \min\{|u(a)|,|u(b)|\}\mu \leq |\lambda| \leq
\max\{|u(a)|,|u(b)|\}\tfrac{1}{\mu}\bigr\},
\]
where $\varphi$ is conjugate to the automorphism
\[
\psi(z) \ = \ \frac{(1 + \mu)z  + (1-\mu)}{(1 - \mu)z  + (1 + \mu)}
\]
for $0 < \mu < 1$. In Theorem \ref{33} we improve this result to obtain an exact expression for the spectral radius.

\begin{lem}\label{23}\cite[Theorem 7]{10}
If $\varphi$ is a univalent selfmap of $\mathbb{D}$, then the spectral radius $r_{\mathcal{D}}(C_{\varphi}) = 1$.
\end{lem}

\begin{thm}\label{33}
Suppose that the weighted composition operator $uC_{\varphi} : \mathcal{D} \rightarrow \mathcal{D}$ is invertible on the Dirichlet space and assume that the automorphism $\varphi$ is hyperbolic, with attractive fixed point $a \in \partial \mathbb{D}$ and repulsive fixed point $b \in \partial\mathbb{D}$. If $u \in A(\mathbb{D})$, then $r_{\mathcal{D}}(uC_{\varphi}) \ = \ \max\{|u(a)|,|u(b)|\}$ and
\begin{equation*}
\sigma_{\mathcal{D}}(uC_{\varphi}) \ \subseteq \ \big\{\lambda \in \mathbb{C} :  \min\{|u(a)|,|u(b)|\} \leq |\lambda| \leq \max\{|u(a)|,|u(b)|\}\big\}.
\end{equation*} 
\end{thm}

\begin{proof}
According to Theorem \ref{3}, $u$ belongs to $\mathcal{M}(\mathcal{D})$ and is bounded away from zero on $\mathbb{D}$, from which follows that $u(a),u(b) \neq 0$. Since $r_{\mathcal{D}}(C_{\varphi}) = 1$ by Lemma \ref{23}, we begin by estimating the operator norm of $M_{u_{(n)}} : \mathcal{D} \rightarrow \mathcal{D}$:
\begin{align*}
& \ \ \ \ \ \ \ \ \ \ \ \ \ \ \ \ \ \ \ \ \ \  \ \ \ \ \  \|M_{u_{(n)}}\|_{\mathcal{D} \rightarrow \mathcal{D}} \ = \ \sup_{\|f\|_{\mathcal{D}} \leq 1}{\|u_{(n)}\cdot f\|_{\mathcal{D}}} \\
&= \ \sup_{\|f\|_{\mathcal{D}}\leq1}\left[|u_{(n)}(0)f(0)|^2 \, + \, \int_{\mathbb{D}}{|(u_{(n)}\cdot f)^{\prime}(z)|^2dA(z)} \right]^{\frac{1}{2}} \\ 
&\leq \ \sup_{\|f\|_{\mathcal{D}}\leq1}\left[|u_{(n)}(0)f(0)| \, + \, \left(\int_{\mathbb{D}}{|u_{(n)}^{\prime}(z)f(z) \, + \, u_{(n)}(z)f^{\prime}(z)|^2dA(z)}\right)^{\frac{1}{2}}\right] \\ 
&\leq \ \|u_{(n)}\|_{\infty} \ + \ \sup_{\|f\|_{\mathcal{D}}\leq1}{\left(\int_{\mathbb{D}}{|u_{(n)}^{\prime}(z)f(z)|^2dA(z)}\right)^{\frac{1}{2}}} \ + \\
& \ \ \ \ \, \sup_{\|f\|_{\mathcal{D}}\leq1}{\left(\int_{\mathbb{D}}{|u_{(n)}(z)f^{\prime}(z)|^2dA(z)}\right)^{\frac{1}{2}}} \\
&\leq \ \sup_{\|f\|_{\mathcal{D}}\leq1}{\left(\int_{\mathbb{D}}{|u_{(n)}^{\prime}(z)f(z)|^2dA(z)}\right)^{\frac{1}{2}}} \ + \ 2\|u_{(n)}\|_{\infty},
\end{align*}
where we used the subadditivity of the square root function and the triangle inequality for the $L^2$-norm. Now since
\begin{equation*}
u_{(n)}^{\prime}(z) \ = \ \sum_{j=0}^{n-1}{\frac{u_{(n)}(z)}{u \circ \varphi_j(z)}\cdot (u \circ \varphi_j)^{\prime}(z)},
\end{equation*}
we can continue the above estimation as follows:
\begin{align*}
\|M_{u_{(n)}}\|&_{\mathcal{D} \rightarrow \mathcal{D}} \ \leq \ \sup_{\|f\|_{\mathcal{D}}\leq1}{\sum_{j=0}^{n-1}{\left(\int_{\mathbb{D}}{\Big|\frac{u_{(n)}(z)}{u \circ \varphi_j(z)}\cdot (u \circ \varphi_j)^{\prime}(z)\Big|^2|f(z)|^2dA(z)}\right)^{\frac{1}{2}}}} \\
& \ \ \ \ \ \ \ \ \ \ \ \, + \ 2\|u_{(n)}\|_{\infty} \\
&\leq \ \sum_{j=0}^{n-1}{\Big\|\frac{u_{(n)}}{u \circ \varphi_j}\Big\|_{\infty}}\sup_{\|f\|_{\mathcal{D}}\leq1}{\left(\int_{\mathbb{D}}{|u^{\prime}(\varphi_j(z))|^2|\varphi_j^{\prime}(z)|^2|f(z)|^2dA(z)}\right)^{\frac{1}{2}}} \\
& \ \ \ \ + \ 2\|u_{(n)}\|_{\infty}.
\end{align*}
After substituting $w = \varphi_j(z)$, the above integral takes the form
\begin{align*}
\left(\int_{\mathbb{D}}{|u^{\prime}(w)f(\psi_j(w))|^2dA(w)}\right)^{\frac{1}{2}} \ &= \ \|M_{u^{\prime}}(f \circ \psi_j)\|_{A^2} \\
&\leq \ \|M_{u^{\prime}}\|_{\mathcal{D} \rightarrow A^2}\|C_{\psi_j}\|_{\mathcal{D} \rightarrow \mathcal{D}}\|f\|_{\mathcal{D}},
\end{align*}
where $\psi := \varphi^{-1}$ and the norm $\|M_{u^{\prime}}\|_{\mathcal{D} \rightarrow A^2}$ is finite since $u \in \mathcal{M(\mathcal{D})}$, see the section on multiplier spaces. In \cite[Theorem 7]{10}, Mart\'{i}n and Vukoti\'{c} proved that
\begin{equation*}
\|C_{\psi_j}\|_{\mathcal{D} \rightarrow \mathcal{D}} \ \leq \ \sqrt{2}\big(1 + \varrho(\psi(0),0)j\big)^{\frac{1}{2}}
\end{equation*}
for every $j \in \mathbb{N}$, which combined with the results above gives
\begin{align*}
\|M_{u_{(n)}}\|_{\mathcal{D} \rightarrow \mathcal{D}} \ &\leq \ \sum_{j=0}^{n-1}{\Big\|\frac{u_{(n)}}{u \circ \varphi_j}\Big\|_{\infty}}\|M_{u^{\prime}}\|_{\mathcal{D} \rightarrow A^2}\|C_{\psi_j}\|_{\mathcal{D} \rightarrow \mathcal{D}} \ + \ 2\|u_{(n)}\|_{\infty} \\
&\leq \ \left[\sqrt{2}\|M_{u^{\prime}}\|_{\mathcal{D} \rightarrow A^2}\Big\|\frac{1}{u}\Big\|_{\infty} + 2\right]n\big(1+\varrho(\psi(0),0)n\big)^{\frac{1}{2}}\|u_{(n)}\|_{\infty}.
\end{align*}
Applying this to the spectral radius and using Lemmas \ref{22} and \ref{23} gives
\begin{align*}
&r_{\mathcal{D}}(uC_{\varphi}) \ = \ \lim_{n \rightarrow \infty}\|(uC_{\varphi})^n\|_{\mathcal{D} \rightarrow \mathcal{D}}^{\frac{1}{n}} \ \leq \ \limsup_{n \rightarrow \infty}\|M_{u_{(n)}}\|_{\mathcal{D} \rightarrow \mathcal{D}}^{\frac{1}{n}}\hspace*{0.3mm}r_{\mathcal{D}}(C_{\varphi}) \\
&\leq \ \lim_{n \rightarrow \infty}{\left[\sqrt{2}\|M_{u^{\prime}}\|_{\mathcal{D} \rightarrow A^2}\Big\|\frac{1}{u}\Big\|_{\infty} + 2\right]^{\frac{1}{n}}n^{\frac{1}{n}}\big(1+\varrho(\psi(0),0)n\big)^{\frac{1}{2n}}\|u_{(n)}\|_{\infty}^{\frac{1}{n}}} \\
&= \ \max\{|u(a)|,|u(b)|\}.
\end{align*}
On the other hand, by \cite[Lemma 1]{9}, we have that
\begin{align*}
\|u_{(n)}\|_{\infty}^{\frac{1}{n}} \ &\leq \ \|M_{u_{(n)}}\|_{\mathcal{D} \rightarrow \mathcal{D}}^{\frac{1}{n}} \ = \ \big\|\big(uC_{\varphi}\big)^n\big(C_{\varphi_n}\big)^{-1}\big\|_{\mathcal{D} \rightarrow \mathcal{D}}^{\frac{1}{n}} \\
&\leq \ \big\|\big(uC_{\varphi}\big)^n\big\|_{\mathcal{D} \rightarrow \mathcal{D}}^{\frac{1}{n}}\big\|\big(C_{\varphi^{-1}}\big)^n\big\|_{\mathcal{D} \rightarrow \mathcal{D}}^{\frac{1}{n}},
\end{align*}
so letting $n$ tend to infinity and observing that $r_{\mathcal{D}}\big(C_{\varphi^{-1}}\big) = 1$ (by Lemma \ref{23} since $\varphi^{-1} \in \textnormal{Aut}(\mathbb{D})$) we see that
\begin{equation*}
\max\{|u(a)|,|u(b)|\} \ \leq \ r_{\mathcal{D}}(uC_{\varphi}),
\end{equation*}  
and thus
\begin{equation*}
r_{\mathcal{D}}(uC_{\varphi}) \ = \ \max\{|u(a)|,|u(b)|\}.
\end{equation*}
The statement regarding the spectrum $\sigma_{\mathcal{D}}(uC_{\varphi})$ can now be justified exactly as in Theorem \ref{17}.
\end{proof}

\begin{rem}
As already noted, we were not able to compute the spectrum of invertible weighted composition operators with hyperbolic symbols $\varphi$, neither for the Bloch nor the Dirichlet space, except when $|u(a)| = |u(b)|$ in the Bloch case. However, the following conjecture seems plausible and we leave it as an open problem: 
\end{rem}

\noindent \textit{Suppose that the weighted composition operator} $uC_{\varphi} : \mathcal{X} \rightarrow \mathcal{X}$ \textit{is invertible on} $\mathcal{X}$\textit{, where} $\mathcal{X}$ \textit{is either the Bloch space} $\mathcal{B}$ \textit{or the Dirichlet space} $\mathcal{D}$\textit{, and assume that the automorphism} $\varphi$ \textit{is hyperbolic, with attractive fixed point} $a \in \partial \mathbb{D}$ \textit{and repulsive fixed point} $b \in \partial\mathbb{D}$\textit{. If} $u \in A(\mathbb{D})$\textit{, then}
\begin{equation*}
\sigma_{\mathcal{X}}(uC_{\varphi}) \ = \ \big\{\lambda \in \mathbb{C} :  \min\{|u(a)|,|u(b)|\} \leq |\lambda| \leq \max\{|u(a)|,|u(b)|\}\big\}.
\end{equation*}

\section*{Acknowledgements}
The first author is grateful for the financial support from the Doctoral Network in Information Technologies and Mathematics at \AA bo Akademi University.

\end{document}